\documentclass[twoside,a4paper,reqno,11pt]{amsart} 
\usepackage[top=28mm,right=28mm,bottom=28mm,left=28mm]{geometry}
\usepackage{amsmath,amsthm,amsfonts,amssymb,latexsym,enumerate,hyperref, bm, stmaryrd} 

\newcommand{\bC}{{\mathbf C}}
\newcommand{\bZ}{{\mathbf Z}}

\newcommand{\PPr}{\operatorname{Pr}}

\newcommand{\PSL}{\operatorname{PSL}}
\newcommand{\SL}{\operatorname{SL}}

\newcommand{\bO}{{\mathbf O}}
\newcommand{\bF}{{\mathbf F}}
\newcommand{\leqs}{\leqslant}
\newcommand{\geqs}{\geqslant}

\makeatletter
\newcommand{\imod}[1]{\allowbreak\mkern4mu({\operator@font mod}\,\,#1)}
\makeatother

\newtheorem{thm}{Theorem}[section]
\newtheorem{lem}[thm]{Lemma}

\newtheorem{prop}[thm]{Proposition}
\newtheorem{cor}[thm]{Corollary}

\newtheorem*{thmA}{Theorem A}
\newtheorem*{thmB}{Theorem B}

\newtheorem*{thmC}{Theorem C}
\newtheorem*{thmD}{Theorem D}

\theoremstyle{definition}
\newtheorem{rem}[thm]{Remark}

\newtheorem*{deff}{Definition}
\newtheorem{remk}{Remark}
\newtheorem{ex}[thm]{Example}

\def\cent#1#2{{\bf C}_{#1}(#2)}

\begin{document}

%%%%%%%%%%%%%%%%%%%%%%%%%%%%%%%%%%%%%%%%%%%%%%%%%%%%%%%%%%%%%
\title[On the commuting probability of $p$-elements in a finite group]{On the commuting probability of $p$-elements \\ in a finite group} 
%%%%%%%%%%%%%%%%%%%%%%%%%%%%%%%%%%%%%%%%%%%%%%%%%%%%%%%%%%%%%

\author[T.C. Burness]{Timothy C. Burness}
\address{T.C. Burness, School of Mathematics, University of Bristol, Bristol BS8 1UG, UK}
\email{t.burness@bristol.ac.uk}

\author[R.M. Guralnick]{Robert M. Guralnick}
\address{R.M. Guralnick, Department of Mathematics, University of Southern California, Los Angeles, CA 90089-2532, USA}
\email{guralnic@usc.edu}

\author[A. Moret\'o]{Alexander Moret\'o}
\address{A. Moret\'o, Departament de Matem\`atiques, Universitat de Val\`encia, 46100
  Burjassot, Val\`encia, Spain}
\email{alexander.moreto@uv.es}

\author[G. Navarro]{Gabriel Navarro}
\address{G. Navarro, Departament de Matem\`atiques, Universitat de Val\`encia, 46100
 Burjassot, Val\`encia, Spain}
\email{gabriel@uv.es}
  
\thanks{Burness thanks the Department of Mathematics at the University of Padua for their generous hospitality during a research visit in autumn 2021. Guralnick was partially supported by the NSF grant DMS-1901595 and a Simons Foundation Fellowship 609771. Moret\'o and Navarro are supported by Ministerio de Ciencia
  e Innovaci\'on (Grant PID2019-103854GB-I00 funded by 
  MCIN/AEI/10.13039/501100011033). Moret\'o is also supported by Generalitat Valenciana 
  AICO/2020/298. We thank J. Mart\'{\i}nez and P.H. Tiep for useful conversations on this paper.  
  We thank the referee for their helpful comments.}


\makeatletter
\@namedef{subjclassname@2020}{\textup{2020} Mathematics Subject Classification}
\makeatother


\date{\today}

\begin{abstract}
Let $G$ be a finite group, let $p$ be a prime and let ${\rm Pr}_p(G)$ be the probability that two random $p$-elements of $G$ commute. In this paper we prove that ${\rm Pr}_p(G) > (p^2+p-1)/p^3$ if and only if $G$ has a normal and abelian Sylow $p$-subgroup, which generalizes previous results on the widely studied commuting probability of a finite group. This bound is best possible in the sense that for each prime $p$ there are groups with ${\rm Pr}_p(G) = (p^2+p-1)/p^3$ and we classify all such groups. Our proof is based on bounding the proportion of $p$-elements in $G$ that commute with a fixed $p$-element in $G \setminus \textbf{O}_p(G)$, which in turn relies on recent work of the first two authors on fixed point ratios for finite primitive permutation groups. 
\end{abstract}

\maketitle

\section{Introduction}\label{s:intro}

The \emph{commuting probability} of a finite group $G$ is the probability that two random elements of $G$ commute, namely
\[
{\rm Pr}(G)= \frac{|\{(x, y) \in G\times G \, : \, xy=yx\}|}{|G|^2}.
\] 
A celebrated, but elementary, result of Gustafson \cite{gus} asserts that ${\rm Pr}(G)>5/8$ if and only if $G$ is abelian, which is best possible since ${\rm Pr}(D_8) = 5/8$. This concept has been widely studied in recent years and some natural analogues for infinite groups have also been investigated (see, for instance, \cite{amv,  ebe,  gr,  les, neu, toi}). In addition, the commuting variety of elements in Lie algebras and algebraic groups has been a subject of great interest for several decades. This was originally introduced by Motzkin and Taussky \cite{MT} and further studied by Richardson \cite{Ri}, Ginzburg \cite{Gin}, Premet \cite{Pr} and others.

In this paper, we pursue a local version of Gustafson's theorem, which turns out
to be significantly more challenging.

\begin{deff}
Let $G$ be a finite group, let $p$ be a prime and let $G_p$ be the
set of $p$-elements in $G$ (that is, the set of elements in $G$ of order $p^m$ for some $m \geqs 0$). Then
\[
{\rm Pr}_p(G) = \frac{|\{(x, y) \in G_p \times G_p \,:\, xy=yx\}|}{|G_p|^2}
\]
is the probability that two random $p$-elements of $G$ commute. Note that ${\rm Pr}_p(G) = 1$ if and only if $G$ has a normal and abelian Sylow $p$-subgroup.
\end{deff}

Local versions of the commuting probability have also been studied in the context of algebraic groups and Lie algebras. In particular, Premet \cite{Pr} identified the irreducible components of the commuting variety of nilpotent
elements of a reductive Lie algebra defined over an algebraically closed field of good characteristic (and similarly, as an immediate consequence, for unipotent elements in the corresponding reductive algebraic groups).  The set of commuting $r$-tuples of elements of order $p$ (or commuting nilpotent elements of nilpotence degree $p$ in a $p$-restricted Lie algebra) has also been studied for its connection to problems in representation theory (see \cite{CFP}). 
For finite groups, a generating function is presented in \cite{FG} for counting the number of commuting pairs of $p$-elements in some finite classical groups in good characteristic.

In this paper we consider arbitrary finite groups. Given a prime number $p$, set
\[
f(p) = \frac{p^2+p-1}{p^3}.
\]
Our first main result is the following.

\begin{thmA}\label{thmA}
Let $G$ be a finite group and let $p$ be a prime. Then ${\rm Pr}_p(G)> f(p)$ if and only if $G$ has a normal and abelian Sylow $p$-subgroup. 
\end{thmA}

In particular, if $G$ is a nonabelian finite simple group and $|G|$ is divisible by $p$, then ${\rm Pr}_p(G) \leqs f(p)$. We can say more in this situation. 

\begin{thmB}\label{thmB}
Let $G$ be a nonabelian finite simple group and let $p$ be a prime divisor of $|G|$. Then ${\rm Pr}_p(G) = f(p)$ if and only if $p \geqs 5$ and $G$ is isomorphic to ${\rm PSL}_{2}(p)$.
\end{thmB}

In fact, we can classify all the finite groups $G$ with $G=\bO^{p'}(G)$ and $\Pr_p(G) = f(p)$, where $\bO^{p'}(G)$ is the subgroup of $G$ generated by $G_p$. See Theorem \ref{equality} for a  precise statement. In particular, we observe that there is no 
nonsolvable group with $\Pr_2(G)=5/8$ and no nonsolvable group $G = \bO^{3'}(G)$ with $\Pr_3(G) = 11/27$. In addition, if $G$ is given as in Theorem B with $p$ a fixed prime, then ${\rm Pr}_p(G)$ tends to $0$ as $|G|$ tends to infinity; we refer the reader to the end of Section \ref{s:thmb} for further details. 

Our next result, which may be of independent interest, is a key ingredient in the proof of Theorem A. Recall that $\bO_p(G)$ denotes the largest normal $p$-subgroup of $G$.
 
\begin{thmC} \label{thmC} 
Let $G$ be a finite group and let $p$ be a prime. Then 
\[
\frac{|{\cent Gx}_p|}{|G_p|} \leqs \frac{1}{p}
\]
for every $p$-element $x \in G \setminus{\bO_p(G)}$.
\end{thmC}

This can be extended as follows. 

\begin{thmD}\label{thmD}
Let $G$ be a finite group and let $p$ be a prime.  If $x \in G$ is a $p$-element and
\[
\frac{|{\cent Gx}_p|}{|G_p|} > \frac{1}{p},
\]
then $x \in \bZ(\bO_p(G))$.
\end{thmD} 

\begin{remk}\label{r:1}
It is easy to see that the converse of Theorem D is false. For example, if $G = D_{8(2m+1)}$, then $|{\cent Gx}_2|/|G_2| = 1/(2m+2)$ if $x \in \bZ(\bO_2(G))$ has order $4$. On the other hand, in Examples \ref{e:ex1} ($p$ odd) and \ref{e:ex2} ($p=2$) we present a family of examples $(G,p,x)$, where $x \in \bZ(\bO_p(G))$ is nontrivial and $|{\cent Gx}_p|/|G_p|$ tends to $1$ as $p$ tends to infinity.
\end{remk}

\begin{remk}\label{r:p} 
Let $G$ be a finite group with $\bO_p(G)=1$. Then the conclusions in Theorems A and C are still valid if we work with elements of order $p$ instead of all $p$-elements (with essentially no change in the proofs). And similarly for Theorem \ref{equality}, which includes Theorem B as a special case.
\end{remk}  

The proofs of our main results depend upon the classification of finite simple groups.  However, it is worth noting that our proof of Theorem C 
does not require the  classification if we assume that $x$ normalizes, but does not centralize, some normal
$p'$-subgroup of $G$. This implies that the classification is not required for Theorem A under the assumption that the generalized Fitting subgroup of $G$ is a $p'$-group (and so in particular, if $G$ is $p$-solvable). In order to handle the general case, we use a recent result of the first two authors \cite{BG} on fixed point ratios of elements of prime order in primitive permutation groups (see Theorem \ref{t:bg}). 

\begin{remk}\label{r:2}
Let us observe that
\[
\frac{|\bC_G(x)_p|}{|G_p|} = \frac{\Psi(x)}{\Psi(1)},
\]
where $\Psi$ is the permutation character for the action of $G$ on its $p$-elements by conjugation. In the language of permutation groups, this number coincides with the fixed point ratio of $x$ with respect to this action, which explains why the main theorem of  \cite{BG} will be an important ingredient in the proof of Theorem C. 
\end{remk}

\section{Some preliminary results}\label{s:prel}

For the remainder of this paper, all groups are finite and $p$ is a prime number. We will frequently use the elementary fact that if $G$ is a group and $H,K \leqs G$ are subgroups, then 
\begin{equation}\label{e:one}
|H:H \cap K| \leqs |G:K|
\end{equation}
with equality if and only if $G=HK$.

\begin{lem} \label{quop}
Let $G$ be a finite group and let $N$ be a normal $p$-subgroup.  
\begin{itemize}\addtolength{\itemsep}{0.2\baselineskip}
\item[{\rm (i)}] If $x \in G$ is a $p$-element, then 
\[
|{\cent Gx}_p|/|G_p| \leqs   |{\cent {G/N}{Nx}}_p|/|(G/N)_p|.
\]
\item[{\rm (ii)}]  ${\rm Pr}_p(G) \leqs {\rm Pr}_p(G/N)$.
\end{itemize}
\end{lem}

\begin{proof}  
Both parts quickly follow from the fact that $|G_p| = |(G/N)_p||N|$.
\end{proof}  

\begin{rem}
In the previous lemma, the assumption that $N$ is a $p$-subgroup is essential. For example, there is a semidirect product $G = C_{35}{:}D_{12}$ with a normal subgroup $N$ of order $3$ such that  $G/N = D_{10} \times D_{14}$ and we compute 
\[
{\rm Pr}_2(G) = \frac{211}{1296} > \frac{11}{72} = {\rm Pr}_2(G/N).
\]
(Here $G$ is $\texttt{SmallGroup}(420,30)$ in the \textsf{GAP} Small Groups library \cite{GAP}.) One can check that this is the smallest finite group with ${\rm Pr}_p(G) > {\rm Pr}_p(G/N)$ for some prime $p$.
\end{rem}

There is a special case where quotients by normal subgroups of order prime to $p$ do not change the proportions.  

\begin{lem} \label{quop'} 
Let $G$ be a finite group and let $N$ be a central $p'$-subgroup.  
\begin{itemize}\addtolength{\itemsep}{0.2\baselineskip}
\item[{\rm (i)}] If $x \in G$ is a $p$-element, then 
\[
|{\cent Gx}_p|/|G_p|  =   |{\cent {G/N}{Nx}}_p|/|(G/N)_p|.
\]
\item[{\rm (ii)}]  If $N$ is central in $G$, then ${\rm Pr}_p(G) ={\rm Pr}_p(G/N)$.
\end{itemize}
\end{lem} 

\begin{proof} 
Let $x \in G$ be a $p$-element and suppose that $[x,y] \in N$ for some $y \in G$.   Since $[x,N]=1$ it follows that $[x^p,y]=[x,y]^p$, so $[x,y]=1$. In addition, if $y$ is a $p$-element, then $y$ is the only $p$-element in the coset $Ny$ and so (i) follows.   Now (ii) follows from (i), noting that $G$ and $G/N$ both have the same number of $p$-elements.  
\end{proof}

\begin{lem}\label{action}
Let $P$ be a  $p$-group acting on a $p'$-group $K$ and let $L$ be a $P$-invariant subgroup of $K$.
 If 
 \begin{equation}\label{e:hyp}
 \frac{|\cent KP:\cent LP|}{|K:L|} <1,
 \end{equation}
 then 
 \[
 \frac{|\cent KP:\cent LP|}{|K:L|} \leqs \frac{1}{p+1}.
 \]
\end{lem}
 
\begin{proof}
Let $C=\cent K P$ and note that $K \ne CL$ in view of the inequality in \eqref{e:hyp}. For any prime $q$,
let $L_q$ be a $P$-invariant Sylow $q$-subgroup of $L$, which is contained
in a $P$-invariant Sylow $q$-subgroup $K_q$ of $K$ (see \cite[Corollary 3.25]{Is}).
Thus $K_q\cap L=L_q$.
By coprime action, $C_q:= C \cap K_q$ and $C \cap L_q$ are Sylow $q$-subgroups of
$C$ and $C \cap L$, respectively (see \cite[Lemma 3.32]{Is}, for example). 
In view of \eqref{e:hyp} we have
\[
\prod_q \frac{|C_q: C_q \cap L|}{|K_q:L_q|} = \frac{|C:C \cap L|}{|K:L|} < 1
\]
and we note that 
\[
\frac{|C_q: C_q \cap L|}{|K_q:L_q|}  = \frac{|C_q: C_q \cap L_q|}{|K_q:L_q|} \leqs 1
\]
for every prime $q$ (see \eqref{e:one}). Therefore,   
\[
\frac{|C:C \cap L|}{|K:L|} \leqs \frac{|C_q: C_q \cap L|}{|K_q:L_q|}
\]
for every $q$, so the bound in \eqref{e:hyp} implies that
\[
\frac{|C_q: C_q \cap L|}{|K_q:L_q|}<1
\]
for some $q$. As a consequence, we are free to assume that $K$ is a $q$-group.

Arguing by induction on $|K:L|$, we may assume that $L$ is a maximal $P$-invariant subgroup of $K$. Then $L$ is normal in $K$ and $K/L$ does not have any proper nontrivial $P$-invariant subgroups, whence \eqref{e:hyp} implies that $C=\cent KP=\cent LP$. If $|\cent Ky:\cent Ly|=|K:L|$ for every $y \in P$, then $P$ acts trivially on $K/L$ and thus $K=CL$, which is incompatible with \eqref{e:hyp}. Therefore, we may assume that $P=\langle y \rangle$ is cyclic. Then the action of $P$ on $K/L$ is a Frobenius action, which implies that if $x \in K \setminus L$, then $\{L, Lx^z \,:\, z \in P\}$ is a set of distinct cosets of $L$ in $K$. Therefore $|K:L|\geqs |P|+1 \geqs p+1$, as required.
\end{proof}

Next we record the following well known result. 

\begin{lem}\label{elem}
Let $G$ be a finite group,   let $x,y \in G$ and let $K \leqs G$ be a subgroup  normalized by $x$ and $y$. If  $Kx=Ky$ and $|K|$ is coprime with $o(x)o(y)$,
then $x$ and $y$ are $K$-conjugate.
\end{lem}

\begin{proof}
We may assume $G=K\langle x,y\rangle$ and thus $K$ is normal in $G$.
Since $Kx=Ky$, it follows that $K\langle x\rangle=K\langle y\rangle=G$.
Now, $K \cap \langle x \rangle=K \cap \langle y \rangle = 1$ and we also note that  $o(x)=o(y)$ and $\langle x\rangle$, $\langle y\rangle$
are Hall $\pi$-subgroups of $G$, where $\pi$ is the set of primes
dividing $o(x)$.  By the Schur-Zassenhaus theorem,
we have $\langle x\rangle^k=\langle y\rangle$
for some $k \in K$ and thus $x^k=y^n$ for some integer $n$.
Now, $Ky=Kx=Kx^k=Ky^n$ and $y^ny^{-1} \in K \cap \langle y\rangle=1$, so $y^n=y$ and the result follows.
\end{proof}

We shall need one more well known fact about coprime actions, which follows from \cite[Theorem 3.27]{Is}.  

\begin{lem}\label{actions}
Let $G$ and $A$ be finite groups with coprime orders and suppose that $A$ acts on $G$ by automorphisms. Set $C=\cent GA$.
Then $G =C[A,G]$ and $[A,[A,G]]=[A,G]$.
\end{lem}

\section{Proofs of Theorems C and D}\label{s:thmcd}

In this section we prove Theorems C and D. We begin by handling a special case of Theorem C, which relies on the following proposition. In part (i), we write $(Ky)_p$ for the set of $p$-elements in the coset $Ky$, where $p$ is a fixed prime throughout this section.

\begin{prop} \label{special1}  
Let $G$ be a finite group and let $K$ be a normal $p'$-subgroup of $G$. Let $x \in G$ be an element of order $p$ such that $K=[x,K]$ and let $y \in G$ be a $p$-element with $[x,y]=1$.
\begin{itemize}\addtolength{\itemsep}{0.2\baselineskip}
\item[{\rm (i)}] If $[y,K] \ne 1$, then the proportion of elements in $y^K = (Ky)_p$ which 
commute with $x$ is at most $1/(p+1)$.
 \item[{\rm (ii)}] If $L =\langle K, x \rangle$, then the proportion of $p$-elements in the coset $Ly$ which commute with $x$ is at most $1/p$.
 \end{itemize} 
 \end{prop} 
  
\begin{proof}  
First consider (i). Since $K$ is a $p'$-group, Lemma \ref{elem} implies that $y^K$ is precisely the set of $p$-elements in the coset $Ky$. Next observe that $y^K \cap \cent Gx = (yK)_p \cap  \cent Gx = (Ay)_p$, where $A = \bC_K(x)$, and another application of  Lemma \ref{elem} gives $(Ay)_p = y^A$. Therefore, the proportion of elements in $y^K$ which commute with $x$ is equal to
\begin{equation}\label{e:prop}
\frac{|y^K \cap \cent Gx|}{|y^K|}  = \frac{|\cent Kx:\cent Kx \cap \cent Ky|}{|K:\cent Ky|}.
\end{equation}
If every element in $y^K$ commutes with $x$, then $[y,K] \leqs \cent Kx$. But then the three subgroups lemma implies that $y$ centralizes $[x,K]=K$, which is incompatible with the condition $[y,K] \ne 1$ in (i). Therefore, the proportion in \eqref{e:prop} is less than $1$ and by applying Lemma \ref{action} (with $L=\cent Ky$ and $P = \langle x \rangle$) we deduce that it is at most $1/(p+1)$ as required.

We now prove (ii). For $0 \leqs i < p$, let $a_i$ be the number of $p$-elements in $Kx^iy$ commuting with $x$ and let $b_i = |(Kx^iy)_p|$, so $\alpha = \sum_i a_i /\sum_i b_i$ is the proportion of $p$-elements in $Ly$ which commute with $x$. If $[x^iy,K] \ne 1$ for all $i$, then (i) implies that $a_i/b_i \leqs 1/(p+1)$ and we immediately deduce that $\alpha \leqs 1/(p+1)$. Therefore, we may assume $[y,K] =1$ (otherwise replace $y$ by $x^iy$ for some $i$). For $1 \leqs i < p$ it follows that $[x^iy, K] \ne 1$ (since $[x,K]=K$) and thus $a_i/b_i \leqs 1/(p+1)$. Since $|y^K|=1$ we have $a_0 = b_0 = 1$ and we deduce that 
\[
\alpha \leqs \frac{1}{p+1} + \frac{p}{(p+1)m},
\]
where $m = |(Ly)_p|$. Finally, we note that $b_i \geqs (p+1)a_i \geqs p+1$ for $1 \leqs i < p$ (since $x^iy \in Kx^iy$ is a $p$-element commuting with $x$), so $m \geqs 1+(p-1)(p+1) = p^2$ and we conclude that $\alpha \leqs 1/p$.
\end{proof}

We are now ready to prove a special case of Theorem C.

\begin{thm}\label{thmB1}   
Let $G$ be a finite group and let $x \in G$ be an element of order $p$. If there exists a normal $p'$-subgroup $K$ of $G$ with $[x,K] \ne 1$, then 
\begin{equation}\label{e:two}
\frac{|{\cent Gx}_p|}{|G_p|} \leqs \frac{1}{p}.
\end{equation}
\end{thm} 

\begin{proof}  
By Lemma \ref{quop}, we may assume that $\bO_p(G) = 1$. We can also assume that $G=K\cent Gx$ and we may replace $K$ by any proper normal subgroup of 
$G$ contained in $K$ that does not centralize $x$. In particular, by Lemma \ref{actions}, we can replace $K$ by $[x,K]$ and so we may assume that $K=[x,K]$.  

Set $L = \langle K, x \rangle$ and let $y \in G$ be a $p$-element. It suffices to show that the proportion of $p$-elements in the coset $Ly$ which commute with $x$ is at most $1/p$. Clearly, if no $p$-element in $Ly$ commutes with $x$, then this proportion is $0$, so we may assume $[x,y]=1$. Now apply Proposition \ref{special1}(ii).  
\end{proof}

\begin{rem}
Let $\bF^*(G)$ be the generalized Fitting subgroup of $G$. If $\bF^*(G)$ is a $p'$-group, then $\bO_p(G)=1$ and the statement of Theorem \ref{thmB1} holds for every nontrivial $p$-element $x$ because we can replace $x$ by an element of order $p$ in $\langle x \rangle$. Of course, if the upper bound in \eqref{e:two} holds for all elements in $G$ of order $p$ (modulo $\bO_p(G)$), then the same bound holds for every nontrivial $p$-element in $G$.  
\end{rem}

Recall that if $G$ is a permutation group on a finite set $\Omega$, then the \emph{fixed point ratio} of an element $z \in G$, denoted ${\rm fpr}(z,\Omega)$, is the proportion of points in $\Omega$ fixed by $z$. It is easy to see that if $G$ is transitive and $H$ is a point stabilizer, then
\[
{\rm fpr}(z,\Omega) = \frac{|z^G \cap H|}{|z^G|}.
\]
The following is a simplified version of the main theorem of \cite{BG}. 

\begin{thm}\label{t:bg}
Let $G \leqs {\rm Sym}(\Omega)$ be a finite primitive permutation group with point stabilizer $H$. If $z \in G$ has prime order $p$, then either
\[
{\rm fpr}(z,\Omega) \leqs \frac{1}{p+1},
\]
or one of the following holds (up to permutation isomorphism):
\begin{itemize}\addtolength{\itemsep}{0.2\baselineskip}
\item[{\rm (i)}] $G$ is almost simple and either 

\vspace{1mm}

\begin{itemize}\addtolength{\itemsep}{0.2\baselineskip}
\item[{\rm (a)}] $G = S_n$ or $A_n$ acting on $k$-element subsets of $\{1, \ldots, n\}$ with $1 \leqs k < n/2$; or
\item[{\rm (b)}] $(G,H,z,{\rm fpr}(z,\Omega))$ is known.
\end{itemize}
\item[{\rm (ii)}] $G$ is an affine group, $\bF^*(G)=\bF(G)=(C_p)^d$, $z \in {\rm GL}_{d}(p)$ is a 
transvection and ${\rm fpr}(z,\Omega) = 1/p$.
\item[{\rm (iii)}] $G \leqs A \wr S_t$ is a product type group with its product action on $\Omega = \Gamma^t$ and $z \in A^t \cap G$, where $A \leqs {\rm Sym}(\Gamma)$ is one of the almost simple primitive groups in part (i).  
\end{itemize}
\end{thm}

We will also need the following corollary to Theorem \ref{t:bg} in the almost simple setting (see \cite[Corollary 3]{BG}). Recall that the \emph{socle} of an almost simple group $G$ is its unique minimal normal subgroup, which coincides with $\bF^*(G)$.

\begin{cor}\label{c:as}
Let $G \leqs {\rm Sym}(\Omega)$ be a finite almost simple primitive permutation group with socle $J$. If $z \in G$ has prime order $p$, then either
\[
{\rm fpr}(z,\Omega) \leqs \frac{1}{p},
\]
or one of the following holds (up to permutation isomorphism):
\begin{itemize}\addtolength{\itemsep}{0.2\baselineskip}
\item[{\rm (i)}] $J = A_n$ and $\Omega$ is the set of $k$-element subsets of $\{1, \ldots, n\}$ for some $1 \leqs k < n/2$; 
\item[{\rm (ii)}] $(J,p) = ({\rm PSL}_2(q), q-1)$, $({\rm Sp}_6(2),3)$, $({\rm PSU}_4(2),2)$, $({\rm Sp}_n(2),2)$ or $(\Omega_{n}^{\epsilon}(2),2)$.
\end{itemize}
\end{cor}

We will now use Theorem \ref{t:bg} and Corollary \ref{c:as} to handle two more special cases of Theorem C, which will then be applied to obtain the result in full generality. In the following proposition, the \emph{components} of $K$ are the quasisimple groups referred to in the statement.
 
\begin{prop}\label{p:perm1}
Let $K$ be a central product of quasisimple groups with $\bO_p(K)=1$ and let $x,y \in {\rm Aut}(K)$ be nontrivial $p$-elements such that $x$ does not normalize any component of $K$.  Assume that the simple quotients of the components of $K$ are isomorphic.  Then the proportion of elements in $y^K$ which commute with $x$ is at most $1/(p+1)$.
\end{prop}

\begin{proof}  
We may assume that $[x,y]=1$ and $x$ has order $p$. Let $K_1, \ldots, K_t$ be the components of $K$ and set $L_i = K_i/\bZ(K_i) \cong L$.   Note that $t$ is a multiple of $p$ since $x$ acts fixed point freely on the set of components. We can now view $x$ and $y$ as commuting automorphisms of the direct product $J:=L^t$, with $o(x) = p$ and $o(y) = p^m$ for some $m \geqs 1$. Set $G = \langle J, x, y \rangle \leqs {\rm Aut}(J)$ and note that $J$ is the unique minimal normal subgroup of $G$. Now
\[
\frac{|y^J \cap \cent Gx|}{|y^J|} = \frac{|x^J \cap \cent Gy|}{|x^J|}
\]
and it suffices to show that 
\begin{equation}\label{e:bound}
\frac{|x^J \cap \cent Gy|}{|x^J|}  \leqs \frac{1}{p+1}.
\end{equation}

Let $M$ be a maximal subgroup of $G$ containing $\cent Gy$ and observe that 
$M$ does not contain $J$ since $G = J\cent Gy$. This allows us to view $G$ acting primitively on the set of cosets $\Omega = G/M$ and we note that
\[
\frac{|x^J \cap \cent Gy|}{|x^J|} \leqs \frac{|x^G \cap M|}{|x^G|} = {\rm fpr}(x,\Omega).
\]
Then by applying Theorem \ref{t:bg}, noting that $x \not\in {\rm Aut}(L)^t \cap G$ by hypothesis, it follows that ${\rm fpr}(x,\Omega) \leqs 1/(p+1)$ and thus \eqref{e:bound} holds.
\end{proof}

Next we seek a version of Proposition \ref{p:perm1} in the special case where $K$ is quasisimple (see Propositions \ref{p:lie} and \ref{p:sym}). In order to do this, we will need the following elementary result.

\begin{lem}\label{l:classes} 
Let $G$ be a finite group, let $x \in G \setminus \bO_p(G)$ be a $p$-element and set 
\[
D = \{y \in G \,:\, \mbox{$\langle y \rangle$ is $G$-conjugate to $\langle x \rangle$}\}.
\]
Then $|D| \geqs p^2-1$.
\end{lem}

\begin{proof}  
Without loss of generality, we may assume that $\bO_p(G)=1$ and $x$ has order $p$.
Consider the natural action of $G$ on the set $C$ of conjugates of $\langle x \rangle$ and note that $|D| = (p-1)|C|$, so it suffices to show that $|C| \geqs p+1$. Note that $x$ fixes $\langle x \rangle \in C$, so it has at least one fixed point on $C$. If $x$ acts trivially on $C$, then $x$ centralizes each of its conjugates and thus, by Baer's theorem,  $x \in \bO_p(G)$, which is a contradiction. Therefore, $x$ acts nontrivially on $C$ and we conclude that $|C| \geqs p+1$.
\end{proof}

\begin{rem}
Let $G$, $D$ and $p$ be given as in Lemma \ref{l:classes}. Then $|D|=p^2-1$ if and only if $|G_p|=p^2$ and the groups with this property are determined in Lemma \ref{cyclicsylow}.
\end{rem}

We also need the following result, which is a corollary of Theorem \ref{t:bg}.

\begin{lem}\label{l:fpr}
Let $G$ be an almost simple group with socle $J$ and assume $J$ is not isomorphic to an alternating group. Let $p$ be a prime divisor of $|J|$ and suppose $x \in G$ has order $p$. Then there exists an element $y \in J$ of order $p$ such that  
\[
\frac{|y^G \cap \cent Gx|}{|y^G|} \leqs \frac{1}{p+1}.
\]
\end{lem}

\begin{proof}
We may assume $G = \langle J,x\rangle$ and we may embed $\cent Gx$ in a core-free maximal subgroup $H$ of $G$, so
\[
\frac{|y^G \cap \cent Gx|}{|y^G|} \leqs \frac{|y^G \cap H|}{|y^G|} = {\rm fpr}(y,G/H)
\]
for every element $y \in J$ of order $p$. Clearly, the desired conclusion holds if there exists such an element with ${\rm fpr}(y,G/H) \leqs  1/(p+1)$, so we may assume otherwise, in which case $(G,H,y)$ is one of the special cases arising in part (i)(b) of Theorem \ref{t:bg}. More precisely, \cite[Theorem 1]{BG} implies that either $G$ is a classical group in a subspace action (and the special cases that arise are recorded in \cite[Table 6]{BG}), or $G = {\rm M}_{22}{:}2$, $H = {\rm PSL}_3(4).2_2$ and $p=2$. In the latter case one can check that ${\rm fpr}(y,G/H)  = 3/11$ if $y \in J$ is an involution, so we may assume $G$ is a classical group in a subspace action. We now inspect the cases in \cite[Table 6]{BG}.

If $J$ is a unitary, symplectic or orthogonal group, then it is easy to check that in every case $(G,H)$ there exists an element $y \in J$ of order $p$ such that ${\rm fpr}(y,G/H) \leqs 1/(p+1)$. For example, if $J = {\rm PSp}_n(q)$ with $n \geqs 4$, $H=P_1$ is the stabilizer of a $1$-space and $p=q$, then we can take $y = (J_2^2,J_1^{n-4})$, where $J_i$ denotes a standard unipotent Jordan block of size $i$. 

To complete the proof, let us assume $J = {\rm PSL}_n(q)$ is a linear group and note that $H=P_1$ is the stabilizer of a $1$-space. If $n \geqs 4$ then once again it is straightforward to see that there is an element $y \in J$ of order $p$ with ${\rm fpr}(y,G/H) \leqs 1/(p+1)$, so we may assume $n \in \{2,3\}$. Suppose $n=3$. If $p=q \geqs 3$ then we can choose $y = (J_3)$, while for $q=2$ we must take $y = (J_2,J_1)$ and one can use \textsf{GAP} \cite{GAP} to verify the desired bound in the statement of the lemma. Similarly, if $p=q-1 \geqs 3$ then we can take $y$ to be the image (modulo scalars) of a diagonal matrix $(\omega, \omega^{-1},I_1)$, where $\omega \in \mathbb{F}_q^{\times}$ has order $p$. And if $(q,p)=(3,2)$ then $y = (-I_2,I_1)$ is the only option and the result can be checked using \textsf{GAP}. 

Finally, suppose $J = {\rm PSL}_2(q)$, so $q \geqs 7$ since ${\rm PSL}_2(4)$ and ${\rm PSL}_2(5)$ are both isomorphic to $A_5$. If $p=q-1$ then $|y^G| = q(q+1)$ and $|\cent Gx|<q$, so the desired bound holds. Now assume $q=p$. Here both $x$ and $y$ are regular unipotent elements and we compute $|y^G| = (p^2-1)/2$ and $|y^G \cap \cent Gx| = (p-1)/2$, which implies that
\[
\frac{|y^G \cap \cent Gx|}{|y^G|} = \frac{1}{p+1}.
\]
The result follows.
\end{proof}

\begin{prop} \label{p:lie}  
Let $K$ be a quasisimple group such that $\bO_p(K)=1$ and $K/\bZ(K)$ is not isomorphic to an alternating group. Let $x \in {\rm Aut}(K)$ be a nontrivial $p$-element.
\begin{itemize}\addtolength{\itemsep}{0.2\baselineskip}
\item[{\rm (i)}] There is a normal subset $D$ of nontrivial $p$-elements in $K$ such that $|D| \geqs p^2-1$ and the proportion of elements in $D$ which commute with $x$ is at most $1/(p+1)$.
\item[{\rm (ii)}] Let $y \in {\rm Aut}(K)$ be a nontrivial $p$-element. 

\vspace{1mm}

\begin{itemize}\addtolength{\itemsep}{0.2\baselineskip}
\item[{\rm (a)}] The proportion of elements in $y^K$ which commute with $x$ is at most $1/p$, unless $K=\Omega_{n}^{+}(2)$, $n \geqs 8$, $p=2$ and both $x$ and $y$ are transvections, in which case the proportion is $1/2+1/2(2^{\frac{n}{2}}-1)$.
\item[{\rm (b)}] The proportion of elements in $(Ky)_p$ which commute with $x$ is at most $1/p$.    
\end{itemize}
\end{itemize} 
\end{prop}

\begin{proof}  
We may assume $x$ has order $p$. Set $J= K/\bZ(K)$ and view $x$ as an automorphism of $J$ of order $p$. Set $G = \langle J,x\rangle$. By Lemma \ref{l:fpr}, there exists an element $y \in J$ of order $p$ such that 
\begin{equation}\label{e:bds}
\frac{|y^J \cap \cent Gx|}{|y^J|} \leqs \frac{|y^G \cap \cent Gx|}{|y^G|} 
\leqs \frac{1}{p+1}.
\end{equation}
If we write $y$ for the corresponding element in $K$, then by applying Lemma \ref{l:classes} we deduce that the normal subset 
\[
D = \{z \in K \,:\, \mbox{$\langle z \rangle$ is $K$-conjugate to $\langle y \rangle$}\} = \bigcup_{i=1}^tz_i^K
\]
contains at least $p^2-1$ elements. Moreover, \eqref{e:bds} implies that the proportion of elements in $z_i^K$ which commute with $x$ is at most $1/(p+1)$ for $1 \leqs i \leqs t$ and thus part (i) follows.

Now let us turn to part (ii). We may assume $[x,y]=1$ and we may view $y$ as an automorphism of $J$ with $o(y) = p^a$ for some $a \geqs 1$. Set $G = \langle J, x,y \rangle \leqs {\rm Aut}(J)$ and embed $\cent Gy$ in a core-free maximal subgroup $H$ of $G$, which allows us to view $G$ as an almost simple primitive permutation group on $\Omega = G/H$.

For now, let us exclude the special cases $(J,p)$ in Corollary \ref{c:as}(ii). Then 
Corollary \ref{c:as} implies that 
\begin{equation}\label{e:cent}
\frac{|y^J \cap \cent Gx|}{|y^J|} = \frac{|x^J \cap \cent Gy|}{|x^J|} \leqs \frac{|x^G \cap H|}{|x^G|} = {\rm fpr}(x,G/H) \leqs \frac{1}{p}
\end{equation}
and thus the proportion of elements in $y^K$ which commute with $x$ is at most $1/p$. 

Next consider the coset $Ky$. Write $(Ky)_p = y_1^K \cup \cdots \cup y_r^K$ as a disjoint union of $K$-classes. If $Ky \ne K$ then each $y_i$ is a nontrivial $p$-element and so the proportion of elements in $y_i^K$ commuting with $x$ is at most $1/p$ by \eqref{e:cent} and the desired result follows. A very similar argument applies when $Ky = K$, but here we have to account for the identity element. To do this, write $K_p = \{1\} \cup D \cup z_1^K \cup \cdots \cup z_s^K$, where $D$ is the normal subset in (i) and each $z_i$ is nontrivial. Set $a_0 = |D \cap \cent Kx|+1$, $b_0 = |D|+1$, $a_i = |z_i^K \cap \cent Kx|$ and $b_i = |z_i^K|$ for $i \geqs 1$. As above, we have $a_i/b_i \leqs 1/p$ for $i \geqs 1$, so it suffices to show that $a_0/b_0 \leqs 1/p$. If we write $D = y_1^K \cup \cdots \cup y_t^K$, then $|y_i^K \cap \cent Kx|/|y_i^K| \leqs 1/(p+1)$ for each $i$ and thus 
\[
\frac{a_0}{b_0} \leqs \frac{1}{p+1}+\frac{p}{m(p+1)},
\]
where $m = |D|+1$. Since $m \geqs p^2$ we deduce that $a_0/b_0 \leqs 1/p$ and the result follows. 

To complete the proof of (ii), it remains to consider the special cases $(J,p)$ in Corollary \ref{c:as}(ii). In each of these cases, $G$ is an almost simple classical group in a subspace action with point stabilizer $H$ and there exists an element $z \in G$ of order $p$ with ${\rm fpr}(z,G/H)>1/p$. The possibilities for $(G,H,z)$ are recorded in \cite[Table 1]{BG}. By inspection, we observe that either
\begin{itemize}\addtolength{\itemsep}{0.2\baselineskip}
\item[{\rm (a)}] $\cent Gz$ is contained in a maximal subgroup $M$ of $G$ such that ${\rm fpr}(z',G/M) \leqs 1/p$ for all $z' \in G$ of order $p$; or
\item[{\rm (b)}] $G={\rm O}^+_n(2)$, $n \geqs 8$, $p=2$, $H$ is the stabilizer of a nonsingular $1$-space and $z = (J_2,J_1^{n-2})$.
\end{itemize}
So excluding the special case in (b), the previous argument goes through. In particular, the previous argument applies if $y \in K$ (note that in case (b), $z$ is contained in ${\rm O}_n^{+}(2) \setminus J$). 

We have now reduced to the case where $G = {\rm O}_n^{+}(2)$, $p=2$ and both $x$ and $y$ are transvections. Here $y^J = y^G$ and $\cent Gx = H$ is the stabilizer of a nonsingular $1$-space, so \cite[Theorem 1]{BG} gives
\[
\frac{|y^J \cap \cent Gx|}{|y^J|} = {\rm fpr}(y,G/H) = \frac{1}{2}+ \frac{1}{2(2^{\frac{n}{2}}-1)}
\]
for the proportion of elements in $y^K$ commuting with $x$. So this is an exception to the main bound in (ii)(a), but we still claim that the proportion of $2$-elements in $Ky$ commuting with $x$ is at most $1/2$. 

To see this, write $(Ky)_2 = y^K \cup y_1^K \cup \cdots \cup y_r^K$ as a disjoint union. By \cite[Theorem 1]{BG}, the proportion of elements in $y_i^K$ which commute with $x$ is at most $1/3$ for each $1 \leqs i \leqs r$. As a consequence, we deduce that the proportion of $2$-elements in $Ky$ commuting with $x$ is at most $1/2$ so long as 
\[
3.2^{\frac{n}{2}-1} = \frac{3|y^K|}{2^{\frac{n}{2}}-1} \leqs \sum_{i=1}^{r}|y_i^K|.
\]
But this inequality clearly holds since $|z^G| \geqs 2^{\frac{n}{2}-1}(2^{\frac{n}{2}}-1)$ for every nontrivial $2$-element $z \in G$. 
\end{proof} 

\begin{rem}
Let us observe that the upper bound in Proposition \ref{p:lie}(ii)(b) is best possible. For example, let $K = {\rm PSL}_2(p)$ and let $x$ and $y$ 
be inner automorphisms of $K$ of order $p$. Then $|(Ky)_p| = p^2$ and $|\cent Kx|=p$, so the relevant proportion is exactly $1/p$.
\end{rem}

We need a different result to handle alternating and symmetric groups.  

\begin{prop}\label{p:sym}  
Let $L = S_n$ and $J = A_n$, where $n \geqs 5$. Let $x \in L$ be an element of prime order $p$ and let $y \in L$ be a transposition.
\begin{itemize}\addtolength{\itemsep}{0.2\baselineskip}
\item[{\rm (i)}] If $p$ is odd, then the proportion of $p$-elements in $J$ which
commute with $x$ is at most $1/p$. 
\item[{\rm (ii)}]  If $p=2$ and $x$ is not a transposition, then the proportion of $2$-elements in $J$ or $Jy$ which commute with $x$ is at most $1/2$.
\item[{\rm (iii)}] If $p=2$ and $x$ is a transposition, then the proportion of $2$-elements in $L$ which commute with $x$ is at most $1/2$.  
\end{itemize}
\end{prop}

\begin{proof}   
First assume $p$ is odd and $|{\rm supp}(x)|= m \geqs 5$ with respect to the natural action of $L$ on $\{1, \ldots, n\}$. Note that $p$ divides $m$ and it suffices to work inside $H:=A_m \times A_{n-m}$ since $H$ contains $\cent Jx_p$. In particular, we may assume that $m=n$. The result is clear if $m=p$, so assume $m > p$. Here $\cent Jx$ is contained in an imprimitive subgroup $K = A_p \wr A_{m/p}$ and so if we write $J_p = x_0^J \cup \cdots \cup x_t^J$ with $x_0 = 1$, then
\[
\frac{|\cent Jx_p|}{|J_p|} \leqs \frac{|K_p|}{|J_p|} =  \frac{1+\sum_{i=1}^ta_i}{1+\sum_{i=1}^tb_i},
\]
where $a_i = |x_i^J \cap K|$ and $b_i = |x_i^J|$. Now Theorem \ref{t:bg} implies that $a_i/b_i \leqs 1/(p+1)$ for all $i$ and we deduce that  
\[
\frac{|\cent Jx_p|}{|J_p|} \leqs \frac{1}{p+1}+\frac{p}{(p+1)c},
\]
where $c = |J_p|$. Since $c \geqs p^2$ (for example, this follows from Lemma \ref{l:classes}) we conclude that this proportion is at most $1/p$, as required.

So to complete the proof of (i), it remains to handle the special case where $x = (1,2,3)$ is a $3$-cycle. Set $d = |(S_{n-3})_3|$ and note that $|\cent Jx_3| = 3d$. If $a = (1,2,i) \in J$ with $i \geqs 4$, then for each $3$-element $b \in J$ fixing $1$, $2$ and $i$ we see that $a^{\pm}b \in J \setminus \cent Jx$ is a $3$-element. Therefore, $|J_3| \geqs 2d(n-3)+3d$ and thus
\[
\frac{|\cent Jx_3|}{|J_3|} \leqs \frac{3}{2n-3} \leqs \frac{1}{3}
\]
for  $n \geqs 6$. The case $n=5$ can be handled directly.

For the remainder, let us assume $p =2$ and write $|{\rm supp}(x)| = 2m$. For 
$m \geqs 4$ we can essentially repeat the argument in (i). Write $\cent Lx = (S_{2}\wr S_m) \times S_{n-2m}$ and let $a_1$ and $a_2$ be the number of even and odd $2$-elements in $S_2 \wr S_m$, respectively. Similarly, let $b_1 = |(A_{n-2m})_2|$ and $b_2 = |(S_{n-2m} \setminus A_{n-2m})_2|$. Then 
\[
|\cent Jx_2| = a_1b_1+a_2b_2,\;\; |\cent Lx_2 \cap Jy| = a_1b_2+a_2b_1.
\]
We claim that 
\begin{equation}\label{e:ineq}
a_1 \leqs \frac{1}{2}|(A_{2m})_2|,\;\; a_2 \leqs \frac{1}{2}|(S_{2m}\setminus A_{2m})_2|.
\end{equation}
To see this, set $K = A_{2m}$ and $H = (S_2 \wr S_m) \cap K$, so $a_1 = |H_2|$. By \cite[Theorem 1]{BG}, we observe that $|z^K \cap H|/|z^K| \leqs 1/3$ for every nontrivial $2$-element $z \in K$ and by arguing as in case (i) we deduce that $|H_2|/|K_2| \leqs 1/2$. This justifies the first inequality in \eqref{e:ineq} and a very similar argument establishes the second. As an immediate consequence, we deduce that 
\[
|\cent Jx_2| \leqs \frac{1}{2}|(S_{2m}\times S_{n-2m})_2 \cap J|,\;\; |\cent Lx_2 \cap Jy| \leqs \frac{1}{2}|(S_{2m}\times S_{n-2m})_2 \cap Jy|
\]
and thus the proportion of $2$-elements in $J$ and $Jy$ commuting with $x$ is at most $1/2$. In the same way, if $m=3$ then we can reduce to the case $n = 6$ and here we can check the result directly.  

Next assume $m=2$, say $x = (1,2)(3,4)$. Set $d =|(S_{n-4})_2|$ and note that 
\[
|\cent Jx_2| = |\cent Lx_2 \cap Jy| = 4d.
\]
Fix  $i, j$ with $4 < i < j$.   
Let $Z(i,j)$ (respectively, $W(i,j)$) be the set of elements in $L$ of the form 
$uv$, where $u$ is a $4$-cycle (respectively, a double transposition different from $(1,2)(i,j)$) on $\{1,2,i,j\}$ and $v$ is a $2$-element fixing each of these $4$ points. Then 
$|Z(i,j)| = 6d$ and $|W(i,j)| = 2d$, so there are at least $2d$ distinct $2$-elements of
each parity in $Z(i,j) \cup W(i,j)$, none of which commute with $x$. Since there are $(n-4)(n-5)/2$ choices for $\{i,j\}$, and the corresponding sets of $2$-elements are pairwise disjoint, this implies that the proportion of $2$-elements in each coset commuting with $x$ is at most 
\[
\frac{4}{4 + (n-4)(n-5)} \leqs \frac{1}{2}
\]
for $n \geqs 7$. The cases $n=5,6$ can be handled directly.  

Finally, let us assume $x = (1,2)$ is a transposition. Set $d =|(S_{n-2})_2|$ and note that $|\cent Lx_2| = 2d$. For each $j \in \{3, \ldots, n\}$, let $Z_j$ denote the set of $2$-elements in $L$ which interchange $1$ and $j$. Note that the $Z_j$ are pairwise disjoint sets of size $d$ and no element in $Z_j$ commutes with $x$. Therefore, $|L_2| \geqs nd$ and we conclude that the proportion of $2$-elements in $L$ centralizing $x$ is at most $2/n$. 
\end{proof}

\begin{rem}\label{r:transpositions}  
One can show that the conclusion in part (ii) of Proposition \ref{p:sym} also holds when $x$ is a transposition. But the proof is more involved and we do not require the stronger result. 
\end{rem} 

\begin{rem}\label{r:covers}
In the proof of Theorem C, we will need to extend Proposition \ref{p:sym} to central extensions $K$ of $A_n$ with $\bO_p(K)=1$. This follows by Lemma \ref{quop'} unless an element of order $p$ does not centralize $\bZ(K)$. This only occurs when $p=2$ and $K$ is a $3$-fold cover of $A_n$. One can check the result directly for these cases. 
\end{rem}

Finally, we are now ready to complete the proof of Theorem C.  
 
\begin{proof}[Proof of Theorem C] 
Let $G$ be a finite group and let $x \in G \setminus{\bO_p(G)}$ be a $p$-element. Let $\bF(G)$ and $\bF^*(G)$ denote the Fitting and generalized Fitting subgroups of $G$, respectively, and note that $x \not\in \bF(G)$. By Lemma \ref{quop}, we may assume that $\bO_p(G)=1$. Without loss of generality, we may assume $o(x) = p$.

If $x$ does not centralize $\bO_{p'}(G)$,  then the result follows by Theorem \ref{thmB1}. Therefore, we may assume $x \in \cent G{\bF(G)}$ and thus $G$ is nonsolvable. Since $x$ is not in $\bO_p(G)$, $x$ acts faithfully on $\bF^*(G)$ and therefore it must act faithfully on some subgroup $K$, which is a central product of quasisimple components (each with order divisible by $p$). We may assume that $K$ is a minimal such subgroup, which implies that $G$ acts transitively on the components of $K$. Note that $\bO_p(K)=1$.

We can further assume that $G=K\cent Gx$ since both $G$ and $K\cent Gx$ contain the same number of $p$-elements commuting with $x$. Therefore, $\cent Gx$ acts transitively on the components of $K$, so either
\begin{itemize}\addtolength{\itemsep}{0.2\baselineskip}
\item[{\rm (a)}] every orbit of $x$ on the components of $K$ has size $p$; or
\item[{\rm (b)}] $x$ normalizes each component of $K$, inducing
the same automorphism (up to conjugacy) on each component.
\end{itemize}

For each $y \in \cent Gx_p$ it suffices to show that the proportion of $p$-elements
in the coset $Ky$ which commute with $x$ is at most $1/p$. Fix such an element $y$ and observe that we may assume that $G=\langle K, x, y \rangle$. In addition, by repeating the argument above, we can reduce to the case where $\langle x, y \rangle$ acts transitively on the components of $K$. We now consider cases (a) and (b) in turn.
 
First assume (a) holds, so $x$ does not normalize any component of $K$. Let $z \in Ky$ be a nontrivial $p$-element. Then by Proposition \ref{p:perm1}, the proportion of elements in $z^K$ commuting with $x$ is at most $1/(p+1)$. Therefore, if $Ky \ne K$ then the proportion of $p$-elements in $Ky$ which commute with $x$ is at most $1/(p+1)$. Similarly, if $Ky = K$ then Lemma \ref{l:classes} implies that $|K_p| \geqs p^2$ and by expressing $K_p$ as a union of $K$-classes we quickly deduce that $|\cent Kx_p|/|K_p| \leqs 1/p$ as required.

Finally, let us assume (b) holds, in which case $y$ must act transitively on the components of $K$. If $K$ has two or more components, then $y$ is nontrivial and we can just interchange $x$ and $y$ in the argument above (noting that any element in $Ky$ still acts transitively on the set of components). This allows us to reduce to the case where $K$ is quasisimple. The result now follows by applying Propositions \ref{p:lie}(ii)(b) and \ref{p:sym}, except for the case where $K/\bZ(K) = A_n$ is an alternating group, $p=2$ and $x$ acts as a transposition on $K$.    In this case, set $L =\langle K, x \rangle \cong S_n$ and note that it suffices to show that the proportion of $2$-elements in the coset $Ly$ which commute with $x$ is at most $1/2$. This follows from Proposition \ref{p:sym}(iii). 
\end{proof} 
  
Theorem D now follows by combining Theorem C with the following result.

\begin{prop} \label{opg}   
Let $G$ be a finite group and let $x \in \bO_p(G) \setminus \bZ(\bO_p(G))$. 
Then 
\[
\frac{|{\cent Gx}_p|}{|G_p|} \leqs \frac{1}{p}.
\]
\end{prop}

\begin{proof}  
Set $Q = \bO_p(G)$ and note that we may assume $G = Q\cent Gx$. Let $y \in \cent Gx$ be a $p$-element and note that $(Qy)_p = Qy$. Then the number of elements in the coset $Qy$ commuting with $x$ is equal to $|\cent Qx|$, which is at most $|Q|/p$ since $x \not\in \bZ(Q)$. Therefore, the proportion of $p$-elements in 
$Qy$ commuting with $x$ is at most $1/p$ and the result follows.
\end{proof}  

To close this section, we present a family of examples to show that there exist finite groups $G$ with a $p$-element $x$ such that 
\[
\frac{1}{p}< \frac{|{\cent Gx}_p|}{|G_p|} <1.
\]
Note that Theorem D implies that such an element $x$ must be in $\bZ(\bO_p(G))$. In fact, our examples have the property that this ratio tends to $1$ as $|G|$ tends to infinity. 

We consider the cases $p$ odd and $p=2$ separately.

\begin{ex}\label{e:ex1}
Fix an odd prime $p$ and consider the semidirect product $H = A{:}B$, where $A = (C_p)^3$ is elementary abelian and a generator $b$ for $B = C_p$ acts on $A$ with a single Jordan block. Let $a \in A$ be a generator for $A$ as a module for $B$. Note that $A$ contains a normal subgroup $K$ of $H$ with $|K|=p^2$. Fix an element $x \in K \setminus H$.

Let $r$ be a prime with $r \equiv 1 \imod{p}$ and fix a scalar $\mu \in \mathbb{F}_r^{\times}$ of order $p$. Let $V = (\mathbb{F}_r)^p$ be a $p$-dimensional vector space over $\mathbb{F}_r$ and consider the semidirect product $G = V{:}H$, where $K$ acts trivially on $V$, $a$ acts as $(\mu, \ldots, \mu)$ and $b$ acts via $(1,\mu, \ldots, \mu^{p-1})$. We now compute
\begin{equation}\label{e:exp}
|\cent Gx_p| = (p^3-p^2)r^p +p^2,\;\; |G_p| = (p^3-p^2)r^p + (p^4-p^3)r^{p-1} +p^2.
\end{equation}
This follows by counting the $p$-elements in each coset $Vh$ of $V$, noting that if $h$ centralizes $x$, then the entire coset does as well. Here it is also helpful to observe that  $|(hV)_p| = |h^V|$, where $|h^V| = r^p$ if $h \in A \setminus K$ and $|h^V|=r$ for $h \in H \setminus A$.  

Finally, let us observe that both expressions in \eqref{e:exp} are polynomials in $r$ of degree $p$, with the same leading coefficient, whence $|\cent Gx_p|/|G_p|$ tends to $1$ as $r$ tends to infinity. 
\end{ex}

Similarly, we can present a family of examples for $p=2$.

\begin{ex}\label{e:ex2}
Let $H = \langle a,b \rangle = D_{16}$, where $o(a) = 8$ and $o(b) = 2$. Fix an odd prime $r$ and let $V$ be a $2$-dimensional vector space over $\mathbb{F}_r$. Consider the semidirect product $G = V{:}H$, where $a$ acts as $(-1,-1)$ on $V$ and $b$ acts as $(-1,1)$. Let $x \in H$ be an element of order $4$. Since 
\[
|\cent Gx_2| =  4r^2+4,\;\; |G_2| = 4r^2 + 8r +4,
\]
we conclude that the ratio $|\cent Gx_2|/|G_2|$ tends to $1$ as $r$ tends to infinity.
\end{ex}

\section{Proof of Theorem A}\label{s:thma}
 
In this section we prove Theorem A. We begin with some general observations. As always, $G$ is a finite group and $p$ is a prime. As in Section \ref{s:intro}, we set 
\[
f(p) = \frac{p^2+p-1}{p^3}.
\]
 
\begin{lem}\label{p'}
If $G/N$ is a $p'$-group, then ${\rm Pr}_p(G)= {\rm Pr}_p(\langle G_p \rangle) = {\rm Pr}_p(N)$.
\end{lem}

\begin{proof}
This is clear because $G_p = \langle G_p \rangle_p = N_p$. 
\end{proof}

We need the following elementary generalization of Gustafson's theorem \cite{gus} on the commuting probability ${\rm Pr}(G)$. This result explains the presence of the term $f(p)$ in Theorem A and it extends \cite[Lemma 1.3]{les}.

\begin{lem}\label{gust}
 Let $p$ be the smallest prime divisor of $|G|$. If ${\rm Pr}(G)>f(p)$, then $G$ is abelian.
 \end{lem}
 
 \begin{proof}
 Recall that $\PPr(G)=k(G)/|G|$, where $k(G)$ is the number of conjugacy classes in $G$ (see \cite{gus}). Seeking a contradiction, let us assume $G$ is nonabelian. Let $K_1,\ldots, K_r$ be the non-central conjugacy classes of $G$ and note that $|K_i|\geqs p$ for every $i$, so we have  
 \[
 |G|=|\bZ(G)|+\sum_{i=1}^r|K_i|\geqs |\bZ(G)|+rp
 \]
and thus $r\leqs (|G|-|\bZ(G)|)/p$. Therefore, 
 \[
 k(G)=|\bZ(G)|+r\leqs \left(\frac{p-1}{p}\right)|\bZ(G)|+\frac{|G|}{p}
 \]
 and thus
 \[
 \PPr(G)\leqs \frac{(p-1)/p}{|G:\bZ(G)|}+\frac{1}{p}\leqs \frac{(p-1)/p}{p^2}+\frac{1}{p}= f(p),
 \]
 where we have used the fact that $G/\bZ(G)$ is not cyclic in the last inequality. This is a contradiction.
 \end{proof}

We now prove Theorem A.   Note that if $\bF^*(G)$ is a $p'$-group, then the proof does not require the classification of finite simple groups. 

\begin{proof}[Proof of Theorem A]
Let $G$ be a finite group and let $P$ be a Sylow $p$-subgroup of $G$. As previously noted, if $P$ is both normal and abelian, then ${\rm Pr}_p(G)=1$. 

Now assume ${\rm Pr}_p(G)>f(p)$. We need to show that $P$ is a normal abelian subgroup of $G$. To do this, we first use induction on $|G|$ to show that $P$ is normal.

By Lemma \ref{quop}, ${\rm Pr}_p(G/\bO_p(G))>f(p)$. If $\bO_p(G) \ne 1$, then the inductive hypothesis implies that  $G/\bO_p(G)$ has a normal Sylow $p$-subgroup, so  
$\bO_p(G)$ is a Sylow $p$-subgroup of $G$ and we are done. Now assume $\bO_p(G)=1$.  By Theorem C, we have  
\begin{equation}\label{e:eq1}
{\rm Pr}_p(G)=\frac{1}{|G_p|^2}\sum_{x\in G_p}|\bC_G(x)_p| \leqs \frac{1}{|G_p|}\left(1+\frac{|G_p|-1}{p}\right).
\end{equation}
Consider the real-valued function 
\[
\varphi_p(x)=\frac{1}{x}\left(1+\frac{x-1}{p}\right)=\frac{1}{x}\left(1-\frac{1}{p}\right)+\frac{1}{p},
\]
which  is a decreasing function for  $x>0$. Seeking a contradiction, assume that $P$ is not normal. Then $|G_p|\geqs p^2$ (this is clear if $|P| \geqs p^2$, and for $|P|=p$ it follows from the fact that $G$ has at least $p+1$ Sylow $p$-subgroups by Sylow's theorem). Hence, 
\[
\varphi_p(|G_p|)\leqs \varphi_p(p^2)=\frac{1}{p^2}\left(1-\frac{1}{p}\right)+\frac{1}{p}= f(p)
\]
and we conclude that 
\[
{\rm Pr}_p(G)\leqs \varphi_p(|G_p|)\leqs f(p),
\]
a contradiction. Therefore, $P$ is a normal subgroup of $G$.

Finally, Lemma \ref{p'} yields 
\[
\PPr(\bO_p(G))={\rm Pr}_p(\bO_p(G))={\rm Pr}_p(G)> f(p)
\]
and thus Lemma \ref{gust} implies that $\bO_p(G)$ is abelian. 
\end{proof}

\section{Proof of Theorem B}\label{s:thmb}

In this final section we determine the finite groups $G$ with ${\rm Pr}_p(G) = f(p)$, which will allow us to prove Theorem B as a special case. We will need the following auxiliary result.

\begin{lem}\label{cyclicsylow}  
Let $p$ be a prime and let $G$ be a finite group such that
$G=\bO^{p'}(G)$ and $G$ has a Sylow $p$-subgroup of order $p$.
If $\Pr_p(G) = f(p)$, then either
\begin{itemize}\addtolength{\itemsep}{0.2\baselineskip}
\item[{\rm (i)}] $G$ is isomorphic to ${\rm PSL}_2(p)$ or ${\rm SL}_2(p)$; or
\item[{\rm (ii)}] $p = 2^r-1 \geqs 7$ is a Mersenne prime and $G = (C_2)^r{:}C_p$, where $C_p$ acts as a Singer cycle on $(C_2)^r$.
\end{itemize}
\end{lem} 

\begin{proof}  
If $x \in G$ is a nontrivial $p$-element, then $|\cent Gx_p|=p$ and we easily deduce that ${\rm Pr}_p(G)=f(p)$ if and only if $|G_p|=p^2$, or equivalently if $G$ has precisely
$p+1$ Sylow $p$-subgroups. Let $P$ be a Sylow $p$-subgroup and let $K$ be the largest normal subgroup of $G$ normalizing each Sylow $p$-subgroup of $G$. Then $K$ is a $p'$-group and so $[K, P] \leqs K \cap P=1$. Therefore, $K$ is centralized by every $p$-element in $G$, so the condition $G=\bO^{p'}(G)$ implies that $K \leqs \bZ(G)$ and $G/K$ is a doubly transitive subgroup of $S_{p+1}$. Moreover, each point stabilizer in this action is the normalizer of a Sylow $p$-subgroup and thus $|G/K| \leqs p(p^2-1)$.

If $p=2$, then $G/K \cong S_3$ and it is easy to check that $G = S_3 \cong {\rm PSL}_2(2)$ is the only possibility. Similarly, if $p=3$ then $G/K \cong A_4$ and $G=A_4 \cong {\rm PSL}_2(3)$ or ${\rm SL}_2(3) \cong Q_8{:}C_3$ are the only options. For the remainder we may assume that $p \geqs 5$.   

Suppose $p$ is not a Mersenne prime. Then $G/K$ is nonsolvable and by inspecting the list of doubly transitive groups
\cite[Theorem 5.3]{Ca}  we see that $G/K \cong {\rm PSL}_2(p)$. Since ${\rm PSL}_2(p)$ is perfect and $K$ is central in $G$, by considering the Schur multiplier of ${\rm PSL}_2(p)$ we deduce that $G={\rm PSL}_2(p)$ or ${\rm SL}_2(p)$.   

Finally, let us assume $p = 2^r -1 \geqs 7$ is a Mersenne prime, so $r$ is an odd prime.
If $G/K$ is almost simple, we deduce as above that $G ={\rm PSL}_2(p)$ or ${\rm SL}_2(p)$. 
The other possibility is that $G/K$ has a normal elementary abelian $2$-subgroup of order $p+1=2^r$.   Thus $G/K \leqs {\rm AGL}_r(2)$. Here each element in ${\rm AGL}_r(2)$ of order $p$ corresponds to a Singer cycle in ${\rm GL}_r(2)$ and by considering the overgroups of such elements (noting that $G/K$ is generated by $p$-elements and that a point stabilizer 
has order at most $p(p-1)$) we deduce that $G/K = (C_2)^r{:}C_p$. Since $G$ is the normal closure of $P$, it follows that $K$ is a $2$-group. But since $r$ is an odd prime, we deduce that $K=1$ is the only possibility. 
\end{proof}

We can now classify all the finite groups with ${\rm Pr}_p(G) = f(p)$, which yields Theorem B as an immediate corollary. 

\begin{thm}\label{equality} 
Let $p$ be a prime and $G$ a finite group with $G=\bO^{p'}(G)$ and ${\rm Pr}_p(G) = f(p)$. Let $Q=\bO_p(G)$ and let $k$ be a positive integer. Then one of the following holds:
\begin{itemize}\addtolength{\itemsep}{0.2\baselineskip}
\item[{\rm (i)}] $G$ is a $p$-group with $|G:\bZ(G)|=p^2$. 
\item[{\rm (ii)}] $p \geqs 5$, $Q$ is abelian and $G={\rm SL}_2(p) \times Q$ or ${\rm PSL}_2(p) \times Q$.
\item[{\rm (iii)}] $p = 2^r-1 \geqs 3$ is a Mersenne prime and $G=(C_2)^r{:}C_{p^k} \times A$, where $C_{p^k}$ acts as a Singer cycle of order $p$ on $(C_2)^r$ and $A \leqs Q$ is abelian.
\item[{\rm(iv)}] $p=3$ and $G =Q_8{:}C_{3^k} \times A$, where $A \leqs Q$ is abelian. 
\item[{\rm(v)}] $p=2$ and $G = C_3{:}C_{2^k} \times A$, where $A \leqs Q$ is abelian. 
\end{itemize}
\end{thm}

\begin{proof} 
Set $Q = \bO_p(G)$. If $G/Q$ is abelian, then $G$ is a $p$-group and by arguing as in the proof of Lemma \ref{gust}, we see that ${\rm Pr}_p(G) = f(p)$ if and only if $|G:\bZ(G)|=p^2$.
 
For the remainder, we may assume $G/Q$ is nonabelian and thus ${\rm Pr}_p(G) = {\rm Pr}_p(G/Q)$ by Lemma \ref{quop}. This implies that if $x,y \in G_p$ commute modulo $Q$, then $[x,y]=1$, which in turn implies that $Q \leqs \bZ(G)$.

First assume that $Q=1$. Let $P$ be a Sylow $p$-subgroup of $G$ and let $n_p$ be the number of Sylow $p$-subgroups of $G$.  As noted in the proof of Theorem A (see \eqref{e:eq1}), we have
\[
{\rm Pr}_p(G) \leqs \frac{1}{|G_p|}\left(1+\frac{|G_p|-1}{p}\right).
\]
If $|G_p| > p^2$ then 
\[
{\rm Pr}_p(G) \leqs \frac{p+1}{p^2+1} < \frac{p^2+p-1}{p^3}
\]
so we may assume $|G_p| \leqs p^2$ and thus $P$ is abelian. If $|P|=p^2$, then $P$ is normal and abelian and so ${\rm Pr}_p(G) =1$, a contradiction. So we can reduce further to the case where $|P|=p$ and $n_p = p+1$. Now apply Lemma \ref{cyclicsylow} to conclude.

Finally, let us assume $Q \ne 1$ and note that $G/Q$ is one of the groups described in Lemma \ref{cyclicsylow}. First assume $G/Q$ is nonsolvable, so $p \geqs 5$ and $G$ is a central extension of $\PSL_2(p)$ or $\SL_2(p)$ by $Q$. Here (ii) holds since every $p$-central extension of one of these groups is split. 

Suppose that $p$ is a Mersenne prime with $p \geqslant 7$.    Let $T$ be a Sylow $2$-subgroup of $G$.
Then $T$ is elementary abelian and $K:=TQ = T \times Q$.  Thus, $T$ is normal in $G$ and
$G=TP$ with $P$ a Sylow $p$-subgroup of $G$.  Since $P/K$ has order $p$,  $P$ is abelian and so
$P=\langle x \rangle \times A$ where $A$ is central and $x$ induces an automorphism of order $p$ on $T$,
whence (iii) holds.   If $p=3$,  the same argument applies except that $T$ is either elementary abelian of order
$4$ or a quaternion group of order $8$, leading to (iv).   

If $p=2$, the same argument applies with $T \cong C_3$ a Sylow $3$-subgroup of $G$.   This leads to (v). 
\end{proof}  

\begin{cor} \label{ns}  
Let $G$ be a finite group such that $\Pr_p(G) = f(p)$.
\begin{itemize}\addtolength{\itemsep}{0.2\baselineskip}
\item[{\rm (i)}] If $p=2$, then $G$ is solvable and $\bO^{2'}(G)$ is metabelian.
\item[{\rm (ii)}] If $p=3$, then $\bO^{3'}(G)$ is solvable. 
\end{itemize}
\end{cor} 

Finally, we turn to the asymptotic behaviour of ${\rm Pr}_p(G)$ with respect to a fixed prime $p$ and a sequence of simple groups of order divisible by $p$. Set 
\[
f_p(G) = \max\{ f_p(x) \,:\, 1 \ne x \in G_p \},
\]
where $f_p(x) = |\cent Gx_p|/|G_p|$. Note that
\[
{\rm Pr}_p(G) = \frac{1}{|G_p|}\sum_{x \in G_p} f_p(x) \leqs \frac{1}{|G_p|}+\left(1-\frac{1}{|G_p|}\right)f_p(G).
\]
 
\begin{prop}\label{p:aa}
Fix a prime $p$ and let $G = A_n$ be the alternating group of degree $n$. Then ${\rm Pr}_p(G) \to 0$ as $n \to \infty$.
\end{prop}

\begin{proof}
Since $|G_p|$ tends to infinity with $n$, it suffices to show that $f_p(G)$ tends to $0$. Let $y \in S_n$ be a nontrivial $p$-element. It is a straightforward exercise to check that for $n$ large enough, $|\cent {S_n}y_p|$ is maximal when $y$ is a $p$-cycle. Let us also observe that $S_n$ contains an equal number of even and odd $2$-elements commuting with a given $2$-element $z \in S_n$ (this is because $\bO_2(\cent {S_n}z)$ contains odd permutations when $z$ is nontrivial). Therefore, if $n$ is large enough we have  
\[
f_p(G) \leqs \frac{|\cent Gx_p|}{|G_p|}
\]
with $x = (1,\ldots, p) \in S_n$ a $p$-cycle. For each integer $p<j \leqs n$, let $y_j \in S_n$ be a $p$-cycle   with an orbit $\{1, \ldots, p-1,j\}$ and let $Z_j$ be the set of $p$-elements in $\cent {S_n}{y_j}$ that act nontrivially on $\{1, \ldots, p-1,j\}$.    Note that the $Z_j$ are pairwise disjoint.

If $p$ is odd, then $|Z_j|= (p-1)!|(A_{n-p})_p|$ and we have $|\cent Gx_p| =  p|Z_j|/(p-1)! \geqs 2|Z_j|/3$, whence 
 \[
\frac{|\cent Gx_p|}{|G_p|} \leqs \frac{2}{3(n-p)} 
\]
and this upper bound tends to $0$ as $n$ tends to infinity. Similarly, if $p=2$ then 
\[
|Z_j|= |(S_{n-2})_2| =|\cent Gx_2|
\]
and the result follows. 
\end{proof}

It is possible to establish an analogous result for simple groups of Lie type, but the details are more complicated and they will be given elsewhere. Here we just sketch some of the main ideas.  Fix a prime $p$. Let $G$ be a simple group of Lie type over $\mathbb{F}_q$ of (untwisted) rank $r$ and assume $p$ divides $|G|$.  As before, it suffices to show that $f_p(G) \to 0$ 
as $|G| \rightarrow \infty$.

First suppose that $q$ is increasing. Let $x \in G$ be a nontrivial $p$-element such that $f_p(x) = f_p(G)$ and note that we may assume $x$ has order $p$. Let $y \in G$ be a nontrivial $p$-element and observe that 
\[
\frac{|y^G \cap \cent Gx|}{|y^G|}
\]
is the probability that $x$ commutes with a random conjugate of $y$. By the main theorem of \cite{LS}, this ratio goes to $0$ as $q$ tends to infinity. Since this is true for every nontrivial conjugacy of $p$-elements, and since the number of $p$-elements in $G$ tends to infinity as $q$ increases (recall that we are assuming $p$ divides $|G|$), we conclude that $f_p(G) \rightarrow 0$.   

Now suppose $q$ is fixed and $r$ is increasing, so we may assume $G$ is a classical group and we note that $p$ divides $|G|$ if $r \geqs p$. First assume $p$ divides $q$, so  we are considering unipotent elements. By a result of Steinberg (see \cite[Lemma 2.16]{LSbook}, for example) we have $|G_p| = q^{\dim X - r}$, where $X$ is the ambient simple algebraic group. By inspecting \cite{LSbook}, it is easy to see that $|\cent Gx_p|$ is maximal when $x$ is a long root element and the result follows easily. 

Finally, let us assume that $p$ does not divide $q$ and so $x$ is a semisimple element. This situation is somewhat more complicated, but there are several ways to proceed and much stronger results can be established. For example, \cite[Theorem 16]{BGG} implies that if $p$ is odd and $r>2$ then the probability that two random elements of order $p$ generate $G$ tends to $1$ as $|G|$ tends to infinity (in particular, the probability that two such elements commute tends to $0$). With some additional work, this can be extended to $p$-elements, including the case $p=2$ (of course, a pair of involutions will not generate $G$, but the probability that they commute still goes to $0$ as $r$ increases). This stronger result implies that ${\rm Pr}_p(G) \to 0$ as $r$ tends to infinity.

It is interesting to consider some extensions of this problem. For example, suppose $G$ is a finite group such that $\bO_p(G) = 1$ and $G = \bO^{p'}(G)$. Do we have ${\rm Pr}_p(G) \to 0$ as $|G| \to \infty$?

\end{document}